\documentclass[12pt]{mymathart}
\usepackage[theorems]{mymathmacros}
\usepackage{version,etex}
\excludeversion{suppress}
\usepackage{pictexwd}
\renewcommand{\setminus}{ -  }
\begin{document}

\title[Density of Hyperbolicity and Robust Chaos]{One-Parameter Families of Smooth 
Interval Maps: Density of Hyperbolicity and Robust Chaos}
\author{Sebastian van Strien}
\address{Mathematics Institute, University of Warwick, Coventry CV4 7AL,United Kingdom}
\email{strien@maths.warwick.ac.uk}
\date{\today}


\begin{abstract}
In this note we will discuss the notion of robust chaos, and show that
(i) there are natural one-parameter families with robust chaos and (ii)
hyperbolicity is dense within generic one-parameter families (and
so these families are not robustly chaotic).
\end{abstract}

\maketitle

\section{Statement of Results}
In \cite{banerjee1998robust} the notion of {\em robust chaos} was introduced. 
A family of maps $\{f_t\}_{t\in [0,1]}$ is said to have  robust chaos
(or to be robustly chaotic) if there exists no parameter $t\in [0,1]$ for which the map $f_t$ has a 
periodic attractor. Examples of   families with robust chaos where given in that paper, but in these 
families the maps are non-smooth.
The authors conjectured that robust chaos does not occur within smooth families of intervals maps 
$f_t\colon [0,1]\to [0,1]$.
Contradicting this conjecture, in \cite{andrecut2001robust},  \cite{andrecut2001example}, \cite{lett2001example} 
\cite{elhadj2008robustness} and \cite{aguirregabiria2009robust},  examples
where given of families of smooth one-dimensional maps with robust chaos.
Since there is a huge literature on bifurcations of one-parameter families of dynamical systems
(starting perhaps with, for example,  \cite{MR699057}),
we shall clarify the situation in this note.

\begin{thm}[Robust unimodal families are \lq constant{\rq}]
\label{thm:unimodal}
 If $\{f_t\}$ is a smooth unimodal family with robust chaos, 
then all maps within this family are topologically conjugate.
\end{thm}

So the family of robustly chaotic unimodal maps given in the papers cited above
are all topologically conjugate to each other. That the family is robustly chaotic is therefore not surprising!
For multimodal families this need not be the case:

\begin{thm}[A family of cubic maps with robust chaos]
 \label{thm:example}There exists a one-parameter family $\{f_t\}$ of smooth multimodal interval maps
which is robustly chaotic.
\end{thm}

On the other hand, the above example is special: generic one-parameter families are never robustly chaotic.
In fact, hyperbolicity is dense within such families:

\begin{thm}[Hyperbolicity is dense within generic families, and so only exceptional families are robustly chaotic]
\label{thm:exrobust}
Near any one-parameter family of smooth interval maps there exists
a one-parameter family $\{f_t\}$ of smooth intervals maps for which
\begin{itemize}
\item the number of critical points of each of the maps $f_t$ is bounded;
\item the set of parameters $t$ for which  \begin{itemize}
\item {\em all} critical points of $f_t$ are in basins of periodic attractors (such a map $f_t$ is called hyperbolic), 
\item critical points of $f_t$ are not eventually mapped onto other critical points (such a map $f_t$ is said to satisfy the no-cycle condition),
\end{itemize}
is open and dense.
\end{itemize}

\end{thm}

In particular, a generic family 
$\{f_t\}_{t\in [0,1]}$  is {\em not} robustly chaotic. We should also point out the following two facts , see \cite{MvS_book}:
maps $f_t$ which are hyperbolic and satisfy the no-cycle condition, 
are (i) structurally stable and (ii) Lebesgue almost every $x$ is in the basin of a hyperbolic periodic attractor of $f_t$.

\section{The Proofs}

Let us start with the proof of Theorem~\ref{thm:unimodal}. Take a robustly chaotic  family $\{f_t\}$ 
of unimodal maps $f_t\colon [0,1]\to [0,1]$.
The itinerary of the critical point $c_t$ of $f_t$ can only change as $t$ varies, 
if $f^n_t(c_t)=c_t$ for some $n$. But since $\{f_t\}$ is robustly chaotic, this does not happen.
So $f_t$ has the same kneading invariant for each $t\in [0,1]$.
Since $f_t$ has no periodic attractors at all, it follows from the non-existence of wandering intervals
(see Chapter IV of \cite{MvS_book}) that $f_{t'}$ and $f_{t}$ are topologically conjugate for
all $t,t'\in [0,1]$.

Let us now prove  Theorem~\ref{thm:example} and show that there exists a family of cubic maps with robust chaos and which
does not have constant kneading invariant. Consider polynomials $f\colon [0,1]\to [0,1]$
of degree three, so that $f(0)=0$, $f(1)=1$ (which implies that $f(x)=ax+bx^2+(1-a-b)x^3$)
and with two critical points $0<c_1<c_2<1$ so that
$0<c_1<f(c_2)<f^3(c_2)=f^4(c_2)<c_2<f^2(c_2)<f(c_1)<1$.
The set of such polynomials corresponds
to a real analytic curve in the $(a,b)$ plane (defined by the condition that $f^4(c_2)=f^3(c_2)$).
Hence it contains a one-parameter family of maps $\{f_t\}_{t\in [0,1]}$. Since $f_t$ is a polynomial with 
only real critical points, it  has negative Schwarzian (see \cite[Exercise IV.1.7]{MvS_book}).
Hence by Singer's result, each of  its periodic attractors has a critical point
in its immediate basin. Since $[f(c_2),1]$ is mapped into itself, and $f(c_1)\in [f(c_2),1]$
any periodic attractor of $f_t$ would have to lie in $[f(c_2),1]$. Since $c_2$ is the only
critical point in $[f(c_1),1]$, it follows that if $f_t$ has a periodic attractor, then
$c_2$ would have to be in its basin.  But since $f^4(c_2)=f^3(c_2)$ is a repelling fixed point, this does not happen.
It follows that these maps define a one-parameter family $\{f_t\}$ of smooth bimodal maps
which are robustly chaotic. Since $f_t(c_1)$ can vary with $t$ (to be anywhere within the interval $[f_t)^2(c_2),1]$), the kneading invariant of $f_t$ is not
constant.
Note that the example is based on the map having a trapping region. 

Let us finally prove Theorem~\ref{thm:exrobust}. Take a one-parameter $\{f_t\}_{t\in [0,1]}$
family of real polynomial interval maps of degree $d$. 
By taking $d$ large enough, we can take this family arbitrarily close  to the original family of interval maps
(in any topology). Let $P$ be the space of  all real polynomial interval maps
of degree $d$.
By \cite{MR2335796} (which is based on \cite{MR2342693}) each 
map $g\in P$ can be approximated by a map $\hat g$ for which all critical points
are in basins of periodic attractors. Hence we can identify
$P$ with $\R^n$, $\{f_t\}$ with a curve $c\colon [0,1]\to \R^n$
and the set of maps in $P$ for which all critical points
are in basins of periodic attractors with an open and dense subset $X$ of $\R^n$.
Maps which fail the no-cycle condition correspond to maps for which an iterate
of a critical point lands on another critical point; the corresponding parameters lie
on analytic codimension-one varieties. So we can and will assume that $X$ corresponds
to hyperbolic maps for which the no-cycle conditions holds.
Hence Theorem~\ref{thm:exrobust} follows from

\begin{lem}
Let $c\colon [0,1]\to \R^n$ be a curve, and let $X$ be an open and dense subset of $\R^n$.
Then there exist a set $A\subset \R^n$ which is dense (in fact of 2nd Baire category)
so that for each  $\alpha\in A$,
$F_\alpha:=\{t\in [0,1]; c(t)+\alpha\in X\}$ is open and dense.
\end{lem}
\begin{proof}
Since the curve $c$ is continuous and $X$ is open,
$F_\alpha$ is open for each $\alpha\in \R^n$. To prove that $F_\alpha$
is dense,  take $\delta>0$ and define the set
$A_\delta$ of $\alpha\in \R^n$ so that for each
$t\in [0,1]$ there exists $t'$ with $|t-t'|<\delta$ and so that
$t'\in F_\alpha$. 

Let us show that $A_\delta$ is dense. 
Assume by contradiction it is not dense. Then
there exists an open set $U$ of $\alpha\in \R^n$
for which there exists $t_\alpha \in [0,1]$ so that
for each $t\in [0,1]$ with $|t-t_\alpha|<\delta$ one
has $t\notin F_{\alpha}$. So if we take $n>1/\delta$ then for each
$\alpha\in U$ there exists $k\in \{0,1,\dots,n\}$ so that
$k/n \notin F_\alpha$, i.e. $c(k/n)+\alpha\notin X$.
Let $U_k$ be the set of $\alpha\in U$ so that
$c(k/n)+\alpha\notin X$. 
Note that $U_0\cup \dots \cup U_n=U$. It follows that the closure of at least
one of the sets $U_{k_0}$ contains an open set (otherwise $U\setminus \overline U_i$
is dense in $U$ for each $i$, 
and so $\bigcap_{i=0,\dots,n} (U\setminus \overline U_i)=U\setminus \bigcup_{i=0,\dots,n} \overline U_i
$ is dense in $U$, a contradiction).  It follows that there exists a subset $U_{k_0}'\subset U_{k_0}$ 
so that $\overline{U_{k_0}'}$ contains an open set. Since  $U_{k_0}'\subset U_{k_0}$,
for each $\alpha\in U_{k_0}'$ 
one has $c(k_0/n)+\alpha\notin X$. But since $X$ is open then for each $\alpha\in \overline{U_{k_0}'}$
one has  $c(k_0/n)+\alpha\notin X$.
But this contradicts the assumption that $X$ is open and dense.
Thus we have shown that $A_\delta$ is dense for each $\delta>0$.

Since $A_\delta$ is also open, it follows by the Baire property
that $A:=\cap_{\delta>0}A_\delta$ is dense.
By construction, for each $\alpha\in A$, we have that
$F_\alpha$ is dense. 
\end{proof}

\tiny 
\bibliographystyle{alpha}
\bibliography{../onedim_bib}

\end{document}